\documentclass[12pt]{amsart}
\usepackage{amsmath}
\usepackage{amssymb}
\usepackage{nccmath}
\usepackage{tikz}
\usepackage[all]{xy}
\usepackage{longtable}
\allowdisplaybreaks[1]
\newtheorem{theorem}{Theorem}
\newtheorem{lemma}{Lemma}
\newtheorem{proposition}{Proposition}
\newtheorem{corollary}{Corollary}
\newtheorem{conjecture}{Conjecture}

\theoremstyle{definition}

\theoremstyle{remark}
\newtheorem{remark}{Remark}

\DeclareMathAlphabet{\matheur}{U}{eur}{m}{n}

\newcommand{\Km}{\mathrm{Km}}
\newcommand{\T}{\mathrm{T}}
\newcommand{\Sym}{\mathrm{Sym}}

\DeclareMathOperator{\Gal}{Gal}

\newmuskip\pFqskip
\mathchardef\pFcomma=\mathcode`, 

\renewcommand{\Re}{\operatorname{Re}}

\begin{document}

\title{Feynman integrals and critical modular $L$-values}

\author{Detchat Samart}
\address{Department of Mathematics, University of Illinois at Urbana-Champaign, Urbana,
IL 61801, USA} \email{dsamart@illinois.edu}


\date{\today}

\maketitle
\begin{abstract}
Broadhurst \cite{Bro} conjectured that the Feynman integral associated to the polynomial corresponding to $t=1$ in the one-parameter family $(1+x_1+x_2+x_3)(1+x_1^{-1}+x_2^{-1}+x_3^{-1})-t$ is expressible in terms of $L(f,2),$ where $f$ is a cusp form of weight $3$ and level $15$. Bloch, Kerr and Vanhove \cite{BKV} have recently proved that the conjecture holds up to a rational factor. In this paper, we prove that Broadhurst's conjecture is true. Similar identities involving Feynman integrals associated to other polynomials in the same family are also established.
\end{abstract}
\section{Introduction}
In this article, we consider the evaluation of the integral
\begin{equation*}
I(t):=\int_{x_1,x_2,x_3\ge 0} \frac{1}{(1+x_1+x_2+x_3)(1+x_1^{-1}+x_2^{-1}+x_3^{-1})-t}\,\frac{d x_1}{x_1}\frac{d x_2}{x_2}\frac{d x_3}{x_3},
\end{equation*}
which is known as the Feynman integral associated to the three-banana graph. The zero loci of the polynomials $P_t(x_1,x_2,x_3):=(1+x_1+x_2+x_3)(1+x_1^{-1}+x_2^{-1}+x_3^{-1})-t$ constitute a one-parameter family $X_t$ of $K3$ surfaces of generic Picard number $19$. Moreover, there are countably many values of $t$ for which $X_t$ has Picard number $20$, in which case it is said to be \textit{singular}. A singular $K3$ surface over $\mathbb{Q}$ is known to be modular in the sense that its Hasse-Weil $L$-function coincides with the $L$-function attached to a weight $3$ cusp form \cite{Liv}. Examples of these values include $t=-32, -2, 1, 4,$ and $16.$ In general, one can determine the values of $t$ which make $X_t$ singular by checking whether an elliptic curve $E_t$ arising naturally in the corresponding Shioda-Inose structure has complex multiplication. (See Section~\ref{S:SI} for more details.) \\
\indent There has been evidence suggesting that some special values of $I(t)$ can be written in terms of interesting arithmetic quantities like zeta values and special values of modular $L$-functions, presumably corresponding to the $K3$ surfaces $X_t$. This phenomenon has been well predicted by Deligne's conjecture on critical values of $L$-functions \cite{Deligne}. When $t=0,$ the variety $X_t$ is obviously reducible, and it was proved in \cite{BBBG,Bro,BKV} that $I(0)=7\zeta(3).$ More intriguingly, Broadhurst \cite{Bro} verified numerically using high-precision computations that
\begin{equation}\label{E:Bro}
I(1)=\frac{12\pi}{\sqrt{15}}L(f,2),
\end{equation}
where $f$ is the weight $3$ cusp form of level $15$ whose $q$-expansion is given by
\begin{align*}f(q)&=\eta(\tau)\eta(3\tau)\eta(5\tau)\eta(15\tau)\sum_{m,n\in\mathbb{Z}}q^{m^2+mn+4n^2}\\ &=\eta^3(3\tau)\eta^3(5\tau)+\eta^3(\tau)\eta^3(15\tau),
\end{align*}
and $L(f,s)$ is the $L$-function associated to $f.$ Here and throughout $\eta(\tau)$ denotes the usual Dedekind eta function
\[\eta(\tau)=q^{1/24}\prod_{n=1}^\infty (1-q^n),\]
where $q=e^{2\pi i \tau}.$ Bloch, Kerr and Vanhove \cite{BKV} have recently proved that Broadhurst's conjecture is true up to a rational coefficient by realizing $I(1)$ and $L(f,2)$ as periods associated to the differential $2$-form on $X_1$ and showing that the underlying regulator is trivial. Their proof involves a special case of Deligne's conjecture for critical $L$-values, which has been proved by Blasius \cite{Bla}. The main goal of this paper is to prove that \eqref{E:Bro} is true.
\begin{theorem}\label{T:main}
Let $I(t)$ and $f$ be defined as above. Then we have
\[I(1)=\frac{12\pi}{\sqrt{15}}L(f,2).\]
\end{theorem}
In addition, we obtain some new identities involving $I(-32), I(-2), I(4),$ and $I(16).$
\begin{theorem}\label{T:main2}
Let $g$ be the weight $3$ cusp form of level $12$ defined by $g(q)=\eta^3(2\tau)\eta^3(6\tau).$ Then we have
\begin{align}
I(-32)+I(16) &= \frac{36 \pi}{\sqrt{12}}L(g,2), \label{E:I2}\\
I(16) = 2I(4)&= 8(I(-2)-I(-32)). \label{E:I3}
\end{align}
\end{theorem}
The proofs of Theorem~\ref{T:main} and Theorem~\ref{T:main2}, given in Section~\ref{S:main}, depend on the modular realization of the integral $I(t)$ given in \cite{BKV} and recent results on critical $L$-values due to Rogers, Wan, and Zucker \cite{RWZ}. Other supplementary results which are required for our proofs will be proved in Section~\ref{S:eta} and Section~\ref{S:zeta}. In Section~\ref{S:SI}, we discuss a Shioda-Inose structure for the family $X_t$ of $K3$ surfaces and give an example of a potential relationship between the Feynman integral and a critical $L$-value of the symmetric square of the underlying elliptic curve in the S-I structure. Finally, we give some further observations in Section~\ref{S:conclude}.

\section{Special values of the Dedekind eta function and Weber's functions}\label{S:eta}
Here we prove some auxiliary results about values of $\eta(\tau)$ and Weber's modular functions:
\begin{align*}
f_0(\tau) &= e^{-\pi i/24}\frac{\eta\left(\frac{\tau+1}{2}\right)}{\eta(\tau)},\\
f_1(\tau) &= \frac{\eta\left(\frac{\tau}{2}\right)}{\eta(\tau)},\\
f_2(\tau) &= \sqrt{2}\frac{\eta(2\tau)}{\eta(\tau)},
\end{align*}
at certain CM points which will appear in the proofs of Theorem~\ref{T:main} and Theorem~\ref{T:main2} .
\begin{lemma}\label{L:eta}
The following evaluations of $\eta(\tau),f_0(\tau),$ and $f_2(\tau)$ are true:
\begin{align}
\left|\eta\left(\mfrac{3+\sqrt{-15}}{2}\right)\right| &= \left(\mfrac{\Gamma\left(\mfrac{1}{15}\right)\Gamma\left(\mfrac{2}{15}\right)\Gamma\left(\mfrac{4}{15}\right)\Gamma\left(\mfrac{8}{15}\right)}{120\pi^3}\right)^{1/4} e^{-\frac{1}{12}\log\left(\frac{1+\sqrt{5}}{2}\right)}, \label{E:eta15a}\\
\left|\eta\left(\mfrac{3+\sqrt{-15}}{6}\right)\right| &= \left(\mfrac{\Gamma\left(\mfrac{1}{15}\right)\Gamma\left(\mfrac{2}{15}\right)\Gamma\left(\mfrac{4}{15}\right)\Gamma\left(\mfrac{8}{15}\right)}{40\pi^3}\right)^{1/4} e^{\frac{1}{12}\log\left(\frac{1+\sqrt{5}}{2}\right)},\label{E:eta15b}\\
\left|f_2\left(\mfrac{3+\sqrt{-15}}{2}\right)\right| &= \left|f_2\left(\mfrac{3+\sqrt{-15}}{6}\right)\right|^{-1}, \label{E:f2a}\\
\left|\eta\left(\mfrac{3+\sqrt{-3}}{2}\right)\right| &= \mfrac{3^{1/8}}{2\pi} \Gamma\left(\mfrac{1}{3}\right)^{3/2}, \label{E:eta3a}\\
\left|\eta\left(\mfrac{3+\sqrt{-3}}{6}\right)\right| &= \mfrac{3^{3/8}}{2\pi} \Gamma\left(\mfrac{1}{3}\right)^{3/2}, \label{E:eta3b}\\
\left|f_2\left(\mfrac{3+\sqrt{-3}}{2}\right)\right| &= \left|f_2\left(\mfrac{3+\sqrt{-3}}{6}\right)\right| = 2^{1/6}, \label{E:f2b}\\
f_0\left(\mfrac{\sqrt{-3}}{3}\right)&= f_0(\sqrt{-3}) =2^{1/3}.\label{E:f0a}
\end{align}
\end{lemma}
\begin{proof}
A general formula for $\left|\eta\left(\mfrac{b+\sqrt{d}}{2a}\right)\right|$, where $ax^2+bxy+cy^2$ is a positive-definite, primitive, integral, binary quadratic form of fundamental discriminant $d<0$, is given in \cite[Thm. 9.3]{PW}:
\begin{equation}
\begin{aligned}
a^{-1/4}&\left|\eta\left(\mfrac{b+\sqrt{d}}{2a}\right)\right| \\
&= (2\pi|d|)^{-1/4}\left\{\prod_{m=1}^{|d|} \Gamma\left(\mfrac{m}{|d|}\right)^{\left(\frac{d}{m}\right)}\right\}^{\frac{w(d)}{8h(d)}} \exp\left(-\mfrac{\pi w(d) \sqrt{|d|}}{48h(d)}\sum\limits_{\substack{L\in H(d) \\ L\ne I}}f(L,K)l(L,d)m(L,d)\right), \label{E:eta}
\end{aligned}
\end{equation}
where $H(d)$ is the group of equivalence classes of such forms, $h(d)=|H(d)|,$
\[w(d)=\begin{cases}
6,  &\text{if } d=-3,\\
4,  &\text{if } d=-4,\\
2,  &\text{if } d<-4,
\end{cases}\]
$K=[a,b,c]\in H(d)$, and the quantities $f(L,K), l(L,d),$ and $m(L,d)$ are defined therein. We have from \cite[Ex. 9.4]{PW} that $H(-15)=\{I,A\}$, where $I$ and $A$ are classes equivalent to $[1,3,6]$ and $[3,3,2]$, respectively, together with the following information about the relevant quantities in \eqref{E:eta}:
\[f(A,I)=1, \quad f(A,A)=-1, \quad l(A,-15)=\mfrac{8}{15}, \quad m(A,-15)=\mfrac{\sqrt{15}}{2\pi}\log \left(\mfrac{1+\sqrt{5}}{2}\right).\]
By Euler's reflection formula
\[\Gamma(x)\Gamma(1-x) =\frac{\pi}{\sin(\pi x)}, \qquad 0<x<1,\]
we can express the product of gamma values in \eqref{E:eta} when $d=-15$ as
\begin{equation*}
\prod_{m=1}^{15} \Gamma\left(\mfrac{m}{15}\right)^{\left(\frac{-15}{m}\right)} = \left(\mfrac{\Gamma\left(\mfrac{1}{15}\right)\Gamma\left(\mfrac{2}{15}\right)\Gamma\left(\mfrac{4}{15}\right)\Gamma\left(\mfrac{8}{15}\right)}{4\pi^2}\right)^2.
\end{equation*} 
Using \eqref{E:eta} with $[a,b,c]=[1,3,6]$ and $[a,b,c]=[3,3,2]$ gives \eqref{E:eta15a} and \eqref{E:eta15b}. We obtain \eqref{E:eta3a} and \eqref{E:eta3b} in a similar way (note that $H(-3)$ is trivial, so the exponential term in \eqref{E:eta} disappears).\\
\indent To prove \eqref{E:f2a} and \eqref{E:f2b}, we employ a formula established in \cite[\S 10]{MW}:
\begin{equation}\label{E:f2}
\left|f_2\left(\mfrac{b+\sqrt{d}}{2a}\right)\right|=\left(\mfrac{2}{\lambda_2}\right)^{1/4} 2^{\frac{m_2}{4}\frac{1-\left(\frac{d}{2}\right)}{2-\left(\frac{d}{2}\right)}}  e^{E(K,d)-E(M_2,\lambda_2^2 d)}.
\end{equation}
In this formula, the integers $a,b,c,$ and $d$ satisfy the same assumption as in \eqref{E:eta} and $m_2,\lambda_2\in\mathbb{Q}$ and $M_2\in H(\lambda_2^2 d)$ are dependent on $K=[a,b,c].$ The quantity $E(K,d)$ is defined by
\begin{equation}\label{E:E}
E(K,d)= \frac{\pi \sqrt{|d|}w(d)}{48h(d)}\sum_{\substack{L\in H(d)\\ L\ne I}} \chi(L,K)^{-1}\frac{t_1(d)}{j(L,d)}l(L,d),
\end{equation}
where the definition of each component in the summand is given in \cite{MW}. It turns out that knowing only certain values of $\chi(L,K)$ is sufficient for our purposes. When $d=-15$, we have from Theorem 2 in \cite[\S 10]{MW} that $\lambda_2=2$ regardless of the choice of equivalence class $K\in H(-15)$ and the classes $M_2$ corresponding to $I=[1,3,6]$ and $A=[3,3,2]$ are $I'=[1,6,24]$ and $A'=[3,6,8]$. Since $\left(\mfrac{-15}{2}\right)=1,$ the formula \eqref{E:f2} yields
\begin{align*}
\left|f_2\left(\mfrac{3+\sqrt{-15}}{2}\right)\right| &= e^{E(I,-15)-E(I',-60)},\\
\left|f_2\left(\mfrac{3+\sqrt{-15}}{6}\right)\right| &= e^{E(A,-15)-E(A',-60)}.
\end{align*}
Using the definition of $\chi(L,K)$ in \cite[\S 6]{MW}, one finds that $\chi(A,I)=\chi(A',I')=1$ and $\chi(A,A)=\chi(A',A')=-1.$ Since $H(-60)=\{I',A'\}$ and $\chi(L,K)$ is the only term on the right-hand side of \eqref{E:E} that depends on $K$, we have  $E(A,-15)=-E(I,-15)$ and $E(A',-60)=-E(I',-60),$ so \eqref{E:f2a} holds. On the other hand, if $d=-3,$ we have $\lambda_2 =2$, $m_2=1$, $\left(\mfrac{d}{2}\right)=-1$ and the exponential term in \eqref{E:f2} again vanishes due to the triviality of the group $H(d)$. This gives \eqref{E:f2b}.\\
\indent Finally, recall that for each $n\in\mathbb{N},$ Ramanujan's class invariant $G_n$ is defined by
\[G_n = 2^{-1/4}f_0(\sqrt{-n}).\]
Many values of $G_n$ are known for odd values of $n;$ in particular, we have $G_3 = 2^{1/12}$ (see, for example, \cite[p.189]{Berndt}), so $f_0(\sqrt{-3})=2^{1/3}.$ The first equality in \eqref{E:f0a} holds since $f_0(\tau)$ is invariant under the transformation $\tau\rightarrow -1/\tau$ \cite{YZ}.
\end{proof}
\begin{remark} 
In a study of short uniform random walks, Borwein et al. \cite[Thm.5.1]{BSWZ} obtained similar eta function evaluations at points in $\mathbb{Q}(\sqrt{-15})$ using the Chowla-Selberg formula, from which \eqref{E:eta} and \eqref{E:f2} were derived.
\end{remark}

\section{Some identities concerning $\pi^3$ and $\zeta(3)$}\label{S:zeta}
In an attempt to generalize Ramanujan's identity for $\zeta(2k+1)$, where $k$ is a non-zero integer, Grosswald \cite{Gro} defined
\begin{equation*}
F_s(\tau) =  \sum_{n\ge 1} \sigma_{-s}(n) e^{2\pi i n\tau},
\end{equation*}
where \[\sigma_t(n)=\sum_{d|n}d^t.\] It is obvious that $F_s$ can be written as a double series
\begin{equation}\label{E:Fs}
F_s(\tau)= \sum_{m,n\ge 1} \frac{e^{2\pi i mn\tau}}{n^s}.
\end{equation}
Moreover, for any odd integer $s>1$, we have from \cite{GMR} that
\begin{equation}\label{E:int}
F_s(\tau) = \sum_{n\ge 1} \frac{\sigma_s(n)}{n^s}e^{2\pi i n\tau} = \frac{(2\pi i)^s B_{s+1}}{2(s+1)(s-1)!}\int_{\tau}^{i\infty}(E_{s+1}(z)-1)(z-\tau)^{s-1}\,dz,
\end{equation}
where $B_k$ is the $k$th Bernoulli number and, for even $k\ge 2$, $E_k$ is the normalized weight $k$ Eisenstein series
\begin{equation*}
E_k(z) = 1-\frac{2k}{B_k}\sum_{n\ge 1}\sigma_{k-1}(n)q^n, \qquad q =e^{2\pi i z}.
\end{equation*}
In other words, when $s>1$ is odd, $F_s(\tau)$ is an Eichler integral of the Eisenstein series of weight $s+1$. In the proofs of our main theorems the identities \eqref{E:Bro}, \eqref{E:I2}, and \eqref{E:I3} will be rephrased in terms of the function $F_3(\tau)$ evaluated at certain algebraic numbers. We first record some useful transformations for $F_3(\tau)$. Throughout this paper $\mathfrak{H}$ denotes the upper-half plane.
\begin{proposition}\label{P:tran}
For all $\tau\in\mathfrak{H}$, we have
\begin{align}
F_3(\tau) &= F_3(\tau+1), \label{E:tran}\\
F_3(\tau)- \tau^2 F_3(-1/\tau) &= \frac{\zeta(3)(\tau^2-1)}{2}+\frac{(2\pi i)^3}{2\tau}\sum_{j=0}^2\frac{B_{2j}B_{4-2j}}{(2j)!(4-2j)!}\tau^{2j}, \label{E:inv}\\
F_3(\tau+1/2)&= -F_3(\tau)+\frac{9}{4}F_3(2\tau)-\frac{1}{4}F_3(4\tau). \label{E:htran}
\end{align}
\end{proposition}
\begin{proof}
The transformation \eqref{E:tran} is obvious from the definition of $F_3(\tau)$ and \eqref{E:inv} is a special case of Grosswald's formula \cite{Gro,GMR}. Using the integral expression \eqref{E:int} for each term in \eqref{E:htran} and performing a change of variable in each integral, we obtain
\begin{align*}
F_3(\tau+1/2)+&F_3(\tau)-\frac{9}{4}F_3(2\tau)+\frac{1}{4}F_3(4\tau)\\
& = \frac{(2\pi i)^3 B_4}{16}\int_\tau^{i\infty}(E_4(z+1/2)+E_4(z)-18E_4(2z)+16E_4(4z))(z-\tau)^2 \,dz.
\end{align*}
It is easily seen from the $q$-expansion of the Eisenstein series that $E_4(z+1/2)+E_4(z)-18E_4(2z)+16E_4(4z)=0$, so \eqref{E:htran} follows.
\end{proof}

\begin{lemma}\label{L:F3}
The following identities are true:
\begin{align}
F_3\left(\mfrac{\sqrt{-3}}{2}+\mfrac{1}{2}\right) &= \mfrac{\sqrt{3}\pi^3}{90}-\mfrac{\zeta(3)}{2}, \label{E:F3a}\\
F_3\left(\mfrac{\sqrt{-3}}{6}+\mfrac{1}{2}\right) &= \mfrac{7\sqrt{3}\pi^3}{810}-\mfrac{\zeta(3)}{2}, \label{E:F3b}\\
8\left(3F_3\left(\mfrac{\sqrt{-3}}{6}\right)-F_3\left(\mfrac{\sqrt{-3}}{2}\right)\right)-9\left(3F_3\left(\mfrac{\sqrt{-3}}{3}\right)-F_3(\sqrt{-3})\right)&=\mfrac{7\sqrt{3}\pi^3}{135}+\zeta(3),\label{E:F3c}\\
24F_3\left(\mfrac{\sqrt{-15}}{6}+\mfrac{1}{2}\right)-8F_3\left(\mfrac{\sqrt{-15}}{2}+\mfrac{1}{2}\right)-3F_3\left(\mfrac{\sqrt{-15}}{3}\right)+F_3(\sqrt{-15})&=\mfrac{\pi^3}{\sqrt{15}}-7\zeta(3). \label{E:F3d}
\end{align}
\end{lemma}
\begin{proof}
Let $\tau_1=\sqrt{-3}/2+1/2$. Then $-1/\tau_1 = \sqrt{-3}/2-1/2=\tau_1-1,$ so we have from \eqref{E:tran} that \[F_3(-1/\tau_1)=F_3(\tau_1-1)=F_3(\tau_1).\] Together with \eqref{E:inv}, this gives \eqref{E:F3a}. Similarly, if $\tau_2= \sqrt{-3}/6+1/2$, then $-1/\tau_2 = \sqrt{-3}/2-3/2=\tau_1 -2,$ so \eqref{E:F3b} follows from \eqref{E:tran}, \eqref{E:inv}, and \eqref{E:F3a}. \\
\indent Next, applying \eqref{E:htran} to $F_3(\tau_1)$ and $F_3(\tau_2)$ and using the fact that $2\sqrt{-3}$ and $2\sqrt{-3}/3$ are sent to $\sqrt{-3}/6$ and $\sqrt{-3}/2$ respectively under the transformation $\tau\rightarrow -1/\tau$ yield
\begin{align*}
4F_3(\tau_1) &= 9F_3(\sqrt{-3})-4F_3\left(\mfrac{\sqrt{-3}}{2}\right)+12F_3\left(\mfrac{\sqrt{-3}}{6}\right)-\mfrac{41\sqrt{3}\pi^3}{216}+\mfrac{13\zeta(3)}{2},\\
12F_3(\tau_2) &= 27F_3\left(\mfrac{\sqrt{-3}}{3}\right)-12F_3\left(\mfrac{\sqrt{-3}}{6}\right)+4F_3\left(\mfrac{\sqrt{-3}}{2}\right)-\mfrac{17\sqrt{3}\pi^3}{216}+\mfrac{7\zeta(3)}{2}.
\end{align*}
The identity \eqref{E:F3c} follows by subtracting the second equation from the first. \\
\indent The proof of \eqref{E:F3d} is much more involved. We first use the transformations in Proposition~\ref{P:tran} to deduce the following identities:
\begin{align}
24F_3\left(\mfrac{\sqrt{-15}}{6}+\mfrac{1}{2}\right) &=(4\sqrt{-15}-4)F_3\left(\mfrac{\sqrt{-15}}{4}+\mfrac{1}{4}\right)+\left(\mfrac{37\sqrt{15}-81i}{270}\right)\pi^3 
 +(2\sqrt{-15}-14)\zeta(3), \label{E:1}\\
 -3F_3\left(\mfrac{\sqrt{-15}}{3}\right) &= -3F_3\left(\mfrac{\sqrt{-15}}{3}+1\right)\\ 
 &= (2-2\sqrt{-15})F_3\left(\mfrac{\sqrt{-15}}{8}+\mfrac{5}{8}\right)-\left(\mfrac{97\sqrt{15}-621i}{4320}\right)\pi^3 -\left(\mfrac{2\sqrt{-15}-5}{2}\right)\zeta(3), \nonumber \\
F_3\left(\mfrac{\sqrt{-15}}{8}+\mfrac{5}{8}\right) &= -F_3\left(\mfrac{\sqrt{-15}}{8}+\mfrac{1}{8}\right)+\mfrac{9}{4}F_3\left(\mfrac{\sqrt{-15}}{4}+\mfrac{1}{4}\right)-\mfrac{1}{4}F_3\left(\mfrac{\sqrt{-15}}{2}+\mfrac{1}{2}\right), \\
F_3\left(\mfrac{\sqrt{-15}}{8}+\mfrac{1}{8}\right) &= \mfrac{(\sqrt{-15}-7)}{32}F_3\left(\mfrac{\sqrt{-15}}{2}+\mfrac{1}{2}\right)+\left(\mfrac{392\sqrt{15}-72i}{61440}\right)\pi^3+\left(\mfrac{\sqrt{-15}-39}{64}\right)\zeta(3), \\
F_3(\sqrt{-15})&=9F_3\left(\mfrac{\sqrt{-15}}{2}+\mfrac{1}{2}\right)-4F_3\left(\mfrac{\sqrt{-15}}{4}+\mfrac{1}{4}\right)-4F_3\left(\mfrac{\sqrt{-15}}{4}-\mfrac{1}{4}\right). \label{E:2}
\end{align}
We simplify the left-hand side of \eqref{E:F3d} using \eqref{E:1}-\eqref{E:2}. After a computation, the remaining term is a multiple of $F_3(\tau_3)-\tau_3^2 F_3(-1/\tau_3),$ where $\tau_3=(\sqrt{-15}-1)/4.$ Therefore, we can apply \eqref{E:inv} once again in the final step to obtain \eqref{E:F3d}.
\end{proof}

\section{Proofs of the main theorems} \label{S:main}
We can now prove Theorem~\ref{T:main} and Theorem~\ref{T:main2}.
\begin{proof}[Proof of Theorem~\ref{T:main}]
Let us first recall from Theorem~2.3.2 and Lemma~2.4.1 in \cite{BKV} that if we parametrize $t$ by the modular function \begin{equation}\label{E:t}
t(\tau)=-\left(\frac{\eta(\tau)\eta(3\tau)}{\eta(2\tau)\eta(6\tau)}\right)^6,
\end{equation}
then for each $\tau\in \mathfrak{H},$ we can express the integral $I(t)$ as
\begin{equation}\label{E:Imod}
I(t(\tau)) = \varpi_1(\tau)\left(\frac{\tau}{2\pi i}\sum_{\substack{m\in\mathbb{Z}\\ n\ge 1}}\frac{\psi(n)}{n^2}\frac{1}{m^2-n^2\tau^2}-4(2\pi i\tau)^3\right),
\end{equation}
where \[\varpi_1(\tau)=\frac{(\eta(2\tau)\eta(6\tau))^4}{(\eta(\tau)\eta(3\tau))^2}\]
and the value of $\psi(n)$ depends only on $n$ modulo $6$, namely \[\psi(0)=-5760,\,\,\psi(\pm 1) = -48,\,\, \psi(\pm 2)= 720,\,\, \psi(3)=384.\]
It then remains to choose $\tau$ appropriately and evaluate the terms on the right-hand side of \eqref{E:Imod} using the results proven in the previous sections. It is shown in \cite[\S 2.5]{BKV} that if $\tau_1=(-3+\sqrt{-15})/24$, then $t(\tau_1)=1,$ whence
\begin{equation*}
I(1)=\varpi_1(\tau_1)\left(\frac{\tau_1}{2\pi i}\sum_{\substack{m\in\mathbb{Z}\\ n\ge 1}}\frac{\psi(n)}{n^2}\frac{1}{m^2-n^2\tau_1^2}-4(2\pi i\tau_1)^3\right).
\end{equation*}
The authors of \cite{BKV} manipulate the right-hand side of the equation above to get a much simpler expression
\begin{equation}\label{E:I1in}
I(1)=-\frac{(2\pi i)^3}{8\sqrt{-15}}\varpi_2(\tau_2),
\end{equation}
where \[\varpi_2(\tau)=\frac{(\eta(\tau)\eta(3\tau))^4}{(\eta(2\tau)\eta(6\tau))^2}\] and $\tau_2=-1/(6\tau_1) = (3+\sqrt{-15})/6. $ Each step in their argument is straightforward, except the following equation \cite[Eq. 2.5.8]{BKV}:
\begin{equation}\label{E:I1in2}
\sum_{\substack{m\ge 1\\ n\ge 1}}\frac{\psi(n)}{n^2}\left(\frac{1}{24m^2-6mn+n^2}+\frac{1}{24m^2+6mn+n^2}\right)=11\zeta(4),
\end{equation}
which is stated without proof. Since our proof is quite involved, we give the details here. \\
Let $S$ be the sum in \eqref{E:I1in2}. Observe that $\sum_{n\ge 1}\psi(n)/n^2 =0,$ so we can rewrite $S$ as
\begin{align*}
S&= \sum_{\substack{m\ge 1\\ n\ge 1}}\frac{\psi(n)}{n^2}\left(\frac{1}{24m^2-6mn+n^2}+\frac{1}{24m^2+6mn+n^2}-\frac{1}{12m^2}\right)\\
&= \frac{1}{96}\sum_{\substack{m\ge 1\\ n\ge 1}}\frac{\psi(n)}{m^3}\left(\frac{2m-n}{24m^2-6mn+n^2}+\frac{2m+n}{24m^2+6mn+n^2}\right)\\
&= \frac{1}{48}\sum_{\substack{m\ge 1\\ n\in\mathbb{Z}}}\frac{\psi(n)}{m^2}\frac{1}{24m^2+6mn+n^2}+ \frac{1}{96}\sum_{\substack{m\ge 1\\ n\in\mathbb{Z}}}\frac{\psi(n)}{m^3}\frac{n}{24m^2+6mn+n^2}+5\zeta(4).
\end{align*}
Substituting $-n-6m$ for $n$ yields
\[\sum_{\substack{m\ge 1\\ n\in\mathbb{Z}}}\frac{\psi(n)}{m^3}\frac{n}{24m^2+6mn+n^2}=-3\sum_{\substack{m\ge 1\\ n\in\mathbb{Z}}}\frac{\psi(n)}{m^2}\frac{1}{24m^2+6mn+n^2}.\]
Therefore, we have
\begin{equation}\label{E:I1in3}
S=-\frac{1}{96}\sum_{\substack{m\ge 1\\ n\in\mathbb{Z}}}\frac{\psi(n)}{m^2}\frac{1}{24m^2+6mn+n^2}+5\zeta(4).
\end{equation}
Let \[T=-\frac{1}{96}\sum_{\substack{m\ge 1\\ n\in\mathbb{Z}}}\frac{\psi(n)}{m^2}\frac{1}{24m^2+6mn+n^2}.\]
Working modulo $6$, one finds that for any expression $f(m,n)$, if $\sum_{\substack{m\ge 1\\ n\in\mathbb{Z}}}f(m,n)^{-1}$ converges absolutely, then
\begin{equation}\label{E:Q}
-\frac{1}{48}\sum_{\substack{m\ge 1\\ n\in\mathbb{Z}}}\frac{\psi(n)}{f(m,n)}=\sum_{\substack{m\ge 1\\ n\in\mathbb{Z}}}\left(\frac{1}{f(m,n)}-\frac{16}{f(m,2n)}-\frac{9}{f(m,3n)}+\frac{144}{f(m,6n)}\right).
\end{equation}
Hence $T$ can be expressed as
\begin{align*}
T= \frac{1}{2}\sum_{\substack{m\ge 1\\ n\in\mathbb{Z}}}\Bigg(&\frac{1}{m^2(24m^2+6mn+n^2)}-\frac{4}{m^2(6m^2+3mn+n^2)}\\
&\qquad -\frac{3}{m^2(8m^2+6mn+3n^2)}+\frac{12}{m^2(2m^2+3mn+3n^2)}\Bigg).
\end{align*}
Then we use the Poisson summation formula to derive the following identities:
\begin{align*}
\sum_{\substack{m\ge 1\\ n\in\mathbb{Z}}}\frac{1}{m^2(24m^2+6mn+n^2)} &= \frac{\pi}{\sqrt{15}}\sum_{\substack{m\ge 1\\ n\in\mathbb{Z}}}\frac{e^{2\pi i m|n|\sqrt{-15}}}{m^3},\\
\sum_{\substack{m\ge 1\\ n\in\mathbb{Z}}}\frac{1}{m^2(6m^2+3mn+n^2)} &= \frac{2\pi}{\sqrt{15}}\sum_{\substack{m\ge 1\\ n\in\mathbb{Z}}}\frac{e^{2\pi i m|n|\left(\frac{\sqrt{-15}}{2}+\frac{1}{2}\right)}}{m^3},\\
\sum_{\substack{m\ge 1\\ n\in\mathbb{Z}}}\frac{1}{m^2(8m^2+6mn+3n^2)} &= \frac{\pi}{\sqrt{15}}\sum_{\substack{m\ge 1\\ n\in\mathbb{Z}}}\frac{e^{2\pi i m|n|\frac{\sqrt{-15}}{3}}}{m^3},\\
\sum_{\substack{m\ge 1\\ n\in\mathbb{Z}}}\frac{1}{m^2(2m^2+3mn+3n^2)} &= \frac{2\pi}{\sqrt{15}}\sum_{\substack{m\ge 1\\ n\in\mathbb{Z}}}\frac{e^{2\pi i m|n|\left(\frac{\sqrt{-15}}{6}+\frac{1}{2}\right)}}{m^3}.
\end{align*}
By \eqref{E:Fs}, we have 
\[\sum_{\substack{m\ge 1\\ n\in\mathbb{Z}}}\frac{e^{2\pi i m|n|\tau}}{m^3}=2F_3(\tau)+\zeta(3),\]
so it follows from \eqref{E:F3d} that
\[T=\frac{\pi^4}{15}=6\zeta(4).\]
Plugging the value of $T$ into \eqref{E:I1in3}, we obtain $S=11\zeta(4),$ as desired.\\
\indent Now we finish the proof by considering the term on the right-hand side of \eqref{E:I1in}. We first rewrite $\varpi_2(\tau)$ in terms of the Weber function $f_2(\tau)$ defined in Section~\ref{S:eta}:
\begin{equation}\label{E:varpi2}
\varpi_2(\tau)= 4\left(\frac{\eta(\tau)\eta(3\tau)}{f_2(\tau)f_2(3\tau)}\right)^2.
\end{equation}
Since $\varpi_2(\tau_2)$ is real, it follows from \eqref{E:eta15a}, \eqref{E:eta15b}, and \eqref{E:f2a} that
\begin{equation}\label{E:varpi2t2}
\varpi_2(\tau_2)=\mfrac{1}{10\sqrt{3}\pi^3}\Gamma\left(\mfrac{1}{15}\right)\Gamma\left(\mfrac{2}{15}\right)\Gamma\left(\mfrac{4}{15}\right)\Gamma\left(\mfrac{8}{15}\right).
\end{equation}
In the final step, we use a result of Rogers, Wan, and Zucker \cite[Thm. 5]{RWZ}:
\begin{equation*}
L(f,2)=\frac{\Gamma\left(\mfrac{1}{15}\right)\Gamma\left(\mfrac{2}{15}\right)\Gamma\left(\mfrac{4}{15}\right)\Gamma\left(\mfrac{8}{15}\right)}{120\sqrt{3}\pi}
\end{equation*}
to complete the proof.
\end{proof}

\begin{remark}
Applying \eqref{E:Q} to the summation in \eqref{E:I1in2} directly and doing a full partial fraction decomposition lead to some (potentially) interesting identities for the \textit{cotangent Dirichlet series}
\begin{equation*}
\xi_s(\tau)=\sum_{n=1}^\infty \frac{\cot(\pi n \tau)}{n^s},
\end{equation*}
where $s$ is odd. More precisely, \eqref{E:I1in2} is equivalent to
\begin{equation*}
3\xi_3\left(\frac{\sqrt{-15}}{6}\right)-\xi_3\left(\frac{\sqrt{-15}}{2}\right)-24\xi_3\left(\frac{\sqrt{-15}}{12}\right)+8\xi_3\left(\frac{\sqrt{-15}}{4}\right) = \frac{2\pi^3}{\sqrt{15}}i.
\end{equation*}
Using a trigonometric identity, we can rewrite this identity as:
\begin{align*}
3\sum_{\substack{n\ge 1\\n \text{ odd }}}\frac{\tan\left(\pi n \frac{\sqrt{-15}}{6}\right)}{n^3}-\sum_{\substack{n\ge 1\\n \text{ odd }}}\frac{\tan\left(\pi n \frac{\sqrt{-15}}{2}\right)}{n^3}=\frac{\pi^3}{4\sqrt{15}}i.
\end{align*}
The study of trigonometric Dirichlet series has a long history dating back to Ramanujan. Nevertheless, very little is known about their special values, especially at imaginary quadratic irrationalities. For more details about trigonometric Dirichlet series, we refer the reader to \cite{Berndt78} and \cite{Str}.
\end{remark}

\begin{proof}[Proof of Theorem~\ref{T:main2}]
We prove \eqref{E:I2} and \eqref{E:I3} using similar arguments. The CM points involved include
\begin{equation*}
\tau_1=\frac{\sqrt{-3}}{3}, \qquad \tau_2=\frac{-3+\sqrt{-3}}{12}, \qquad \tau_3=\frac{\sqrt{-3}}{6}, \qquad \tau_4=\frac{3+\sqrt{-3}}{6}.
\end{equation*}
Observe that we can rewrite $t(\tau)$ in terms of $f_2(\tau)$:
\[t(\tau)=-\left(\frac{2}{f_2(\tau)f_2(3\tau)}\right)^6.\]
Then, using well-known transformations for $f_0,f_1,$ and $f_2$ under $\tau\rightarrow \tau+1$ and $\tau\rightarrow -1/\tau$ \cite{YZ} together with \eqref{E:f0a}, it is not hard to show that $ t(\tau_1)=-32, t(\tau_2)=4, t(\tau_3)=-2, t(\tau_4)=16,$ and \[\left(\mfrac{1-\sqrt{-3}}{2}\right)\varpi_2(\tau_2)=\varpi_2(\tau_4)=4\varpi_2(\tau_3)=2\varpi_2(\tau_1).\]

For a positive integer $N$, let $w_N$ denote the Fricke involution
\begin{equation*}
w_N = \begin{pmatrix}
0 & \frac{1}{\sqrt{N}}\\
-\sqrt{N} & 0
\end{pmatrix},
\end{equation*}
acting on $\mathfrak{H}$ by fractional linear transformation. Then we have $\tau_3= w_6 \tau_1$ and $\tau_4= w_6\tau_2.$ In addition, the transformation
\[\eta(m w_N\tau)=\sqrt{-\left(\frac{iN}{m}\right)\tau}\cdot \eta\left(\frac{N}{m}\tau\right), \qquad m\in\mathbb{N},\]
yields
\begin{equation}\label{E:w6}
\varpi_1(w_6\tau)=-\frac{3}{4}\tau^2 \varpi_2(\tau).
\end{equation}
\indent Therefore, due to \eqref{E:Imod}, \eqref{E:w6}, and some manipulations similar to those used in \cite[\S 2.5]{BKV}, we have
\begin{align*}
I(16)&=\varpi_2(\tau_4)\left(\frac{2\sqrt{3}\pi^3}{9}-\frac{\sqrt{3}}{64\pi}\sum_{\substack{m\ge 1\\ n\ne 0}}\frac{\psi(n)}{n^2}\frac{1}{3m^2+3mn+n^2}\right),\\
I(4)&=\varpi_2(\tau_4)\left(\frac{2\sqrt{3}\pi^3}{9}-\frac{\sqrt{3}}{4\pi}\sum_{\substack{m\ge 1\\ n\ne 0}}\frac{\psi(n)}{n^2}\frac{1}{12m^2+6mn+n^2}\right),\\
I(-2)&= \varpi_2(\tau_4)\left(\frac{\sqrt{3}\pi^3}{18}+\frac{\sqrt{3}}{4\pi}\sum_{\substack{m\ge 1\\ n\ge 1}}\frac{\psi(n)}{n^2}\frac{1}{12m^2+n^2}\right),\\
I(-32)&= \varpi_2(\tau_4)\left(\frac{\sqrt{3}\pi^3}{18}+\frac{\sqrt{3}}{64\pi}\sum_{\substack{m\ge 1\\ n\ge 1}}\frac{\psi(n)}{n^2}\frac{1}{3m^2+n^2}\right).
\end{align*}
Again, with the aid of \eqref{E:Q}, the Poisson summation formula, and Lemma~\ref{L:F3}, we arrive at the following identities:
\begin{align*}
I(16)&=\varpi_2(\tau_4)\left(\frac{32\sqrt{3}\pi^3}{135}-2\zeta(3)-2\left(3F_3\left(\frac{\sqrt{-3}}{3}\right)-F_3(\sqrt{-3})\right)\right),\\
I(4)&=\varpi_2(\tau_4)\left(\frac{16\sqrt{3}\pi^3}{135}-\zeta(3)-\left(3F_3\left(\frac{\sqrt{-3}}{3}\right)-F_3(\sqrt{-3})\right)\right),\\
I(-2)&=\varpi_2(\tau_4)\left(\frac{23\sqrt{3}\pi^3}{540}+\frac{7}{4}\zeta(3)+\frac{7}{4}\left(3F_3\left(\frac{\sqrt{-3}}{3}\right)-F_3(\sqrt{-3})\right)\right),\\
I(-32)&=\varpi_2(\tau_4)\left(\frac{7\sqrt{3}\pi^3}{540}+2\zeta(3)+2\left(3F_3\left(\frac{\sqrt{-3}}{3}\right)-F_3(\sqrt{-3})\right)\right),
\end{align*}
from which it is obvious that
\begin{equation*}
I(16) = 2I(4)= 8(I(-2)-I(-32)) 
\end{equation*}
and that
\begin{equation*}
I(-32)+I(16) = \frac{\sqrt{3}\pi^3}{4}\varpi_2(\tau_4).
\end{equation*}
Since $\varpi_2(\tau_4)$ is real, it follows from \eqref{E:varpi2}, \eqref{E:eta3a}, \eqref{E:eta3b}, \eqref{E:f2b}, and \cite[Thm. 5]{RWZ} that
\[\varpi_2(\tau_4) = \frac{24}{2^{17/3}}\frac{\Gamma^6\left(\frac{1}{3}\right)}{\pi^4}=\frac{24}{\pi^2}L(g,2), \]
which gives \eqref{E:I2}.
\end{proof}

\section{Shioda-Inose structure for the family $X_t$} \label{S:SI}
$K3$ surfaces with large Picard number are well known as a class of varieties enriched with interesting arithmetic information, usually encoded in their associated $L$-functions. In order to gain more insight into the interplay between the Feynman integral $I(t)$ and special values of $L$-functions, we will investigate certain arithmetic properties of the family $X_t$. Recall from \cite{BKV} that, for all but finitely many $t\in\bar{\mathbb{Q}}$, resolving the singularities of the hypersurface $P_t(x_1,x_2,x_3):=(1+x_1+x_2+x_3)(1+x_1^{-1}+x_2^{-1}+x_3^{-1})-t=0$ yields a $K3$ surface with Picard number at least $19$. Therefore, by a result of Morrison \cite[Cor. 7.4]{Mor}, the family $X_t$ admits a Shioda-Inose structure; i.e., they fit into the following diagram:
\begin{center}
\begin{tikzpicture}
\draw [->](-1,1) -- (-0.5,0) ;
\draw [->] (1,1) -- (0.5,0) ;
\end{tikzpicture}
\put(-65,32.5){$X_t$}
\put(-15,32.5){$E_t\times E_t'$}
\put(-60,-15){$\Km(E_t\times E_t')$}
\end{center}
where $E_t$ and $E_t'$ are isogenous elliptic curves, $\Km(E_t\times E_t')$ is the Kummer surface for $E_t$ and $E_t'$, the two arrows indicate degree two rational maps, and the transcendental lattices of $X_t$ and $\Km(E_t\times E_t')$ are isometric. Moreover, the case when $X_t$ is singular occurs exactly when $E_t$ has complex multiplication. It is also known that the Picard-Fuchs equation of the family $X_t$ is the symmetric square of an order two ordinary linear Fuchsian differential equation, where the latter, up to a change of variables, is the Picard-Fuchs equation for a family of elliptic curves giving rise to a S-I structure on $X_t$ \cite{Dor,Lon}.

We have from \cite[\S 2.3]{BKV} that the holomorphic period
\[u(t)=\frac{1}{(2\pi i)^3}\int_{|x_1|=|x_2|=|x_3|=1}\frac{1}{(1+x_1+x_2+x_3)(1+x_1^{-1}+x_2^{-1}+x_3^{-1})-t}\, \frac{d x_1}{x_1}\frac{d x_2}{x_2}\frac{d x_3}{x_3}\]
satisfies the Picard-Fuchs equation $\mathcal{L}_t^3 u(t) =0,$ where 
\begin{equation}\label{E:L3}
\mathcal{L}_t^3 = t^2(t-4)(t-16)\frac{d^3}{d t^3}+6t(t^2-15t+32)\frac{d^2}{d t^2}+(7t^2-68t+64)\frac{d}{d t}+t-4.
\end{equation}
On the other hand, it is shown in \cite{SB} that the family of elliptic curves 
\[E_r: \quad xy-r(x+y+1)(xy+y+x)=0,\]
has Picard-Fuchs equation 
\begin{equation}\label{E:L2a}
r(r-1)(9r-1)\frac{d^2 v}{d r^2} + (27r^2-20r+1)\frac{d v}{d r} + (9r-3)v =0.
\end{equation}
By the change of variables $w=r^{1/2}v$ and $t-4=-(1-3r)^2/r,$ given in the proof of \cite[Lem. 10]{Ver}, we can transform \eqref{E:L2a} into
\[\mathcal{L}_t^2 w =0,\]
where 
\begin{equation}\label{E:L2b}
\mathcal{L}_t^2= t(t-4)(t-16)\frac{d^2}{d t^2}+2(t^2-15t+32)\frac{d}{d t}+\frac{t-8}{4}.
\end{equation}
By \cite[Prop. 11]{Ver} (or direct computation), $\mathcal{L}_t^3$ is the symmetric square of $\mathcal{L}_t^2,$ so the family $E_{r(t)},$ where $r(t)=(10-t+\sqrt{t^2-20t+64})/18,$ gives rise to the S-I structure for $X_t.$ When $X_t$ and $E_{r(t)}$ are defined over $\mathbb{Q}$, the symmetric square $L$-function of $E_{r(t)}$ is, up to simple factors, the $L$-function attached to the transcendental lattice $\T(X_t)$. In particular, when $X_t$ is singular, the latter becomes the $L$-function of a weight $3$ cusp form. A concrete example of these assertions is the following.
\begin{theorem}\label{T:Sym}
Let $E_{r}$ be an elliptic curve as defined above and let $g$ be the cusp form defined in Theorem~\ref{T:main2}. Then 
\begin{equation}\label{E:Sym}
L(\Sym^2 E_{-1/3},s)=L(\Sym^2 E_{1/3},s)=L(\chi_{-3},s-1)L(g,s),
\end{equation}
where $\chi_{-3}$ is the Dirichlet character associated to $\mathbb{Q}(\sqrt{-3}).$
\end{theorem}
\begin{proof}
Let $K=\mathbb{Q}(\sqrt{-3})$. Note first that $E_{1/3}$ is isomorphic to the conductor $36$ elliptic curve 
\[y^2=x^3+1,\]
which has complex multiplication by $\mathcal{O}_K=\mathbb{Z}[(1+\sqrt{-3})/2]$. 
Let $\Lambda = (3+\sqrt{-3})\subset \mathcal{O}_K$ and let $P(\Lambda)$ be the set of integral ideals of $\mathcal{O}_K$ coprime to $\Lambda.$ Define $\varphi:P(\Lambda)\rightarrow \mathbb{C}^{\times}$ by \[\varphi((m+n\sqrt{-3}))=\chi_{-3}(m)(m+n\sqrt{-3}),\qquad m,n\in\mathbb{Z},\quad m>0.\] Then $\varphi$ extends multiplicatively to a Gr\"{o}ssencharacter on the group $I(\Lambda)$ of fractional ideals of $\mathcal{O}_K$ coprime to $\Lambda.$ Hence, by \cite[Thm. 1.31]{Ono},
\[\psi(q)=\sum_{\mathfrak{a}\in P(\Lambda)}\varphi(\mathfrak{a})q^{N(\mathfrak{a})}\]
is a newform in $\mathrm{S}_2(\Gamma_0(36))$. By a simple manipulation, we have 
\[\psi(q) = \frac{1}{2}\sum_{m,n\in\mathbb{Z}}m\chi_{-3}(m)q^{m^2+3n^2}= q-4q^7+2q^{13}+O(q^{19}).\]
Since there is only one isogeny class in conductor $36$, $\varphi$ is indeed the Gr\"{o}ssencharacter of the elliptic curve $E_{1/3}.$ Now let 
\[\Psi(q)=\sum_{\mathfrak{a}\in P(\Lambda)}\varphi^2(\mathfrak{a})q^{N(\mathfrak{a})},\]
where $\varphi^2$ is the primitive Gr\"{o}ssencharacter attached to the square of $\varphi$. Then $\Psi(q)$ is expressible as
\[\Psi(q)=\frac{1}{2}\sum_{m,n\in\mathbb{Z}}(m^2-3n^2)q^{m^2+3n^2},\]
which coincides with the cusp form $g$ \cite[Lem. 2.3]{Sam}. Now the second equality in \eqref{E:Sym} follows immediately from \cite[Prop. 5.1]{CS}, while the first equality comes from the fact that the curves $E_{-1/3}$ and $E_{1/3}$ are isogenous.
\end{proof}
Observe that the values $r=1/3$ and $r=-1/3$ correspond to $t=4$ and $t=16$.  As a consequence of Theorem~\ref{T:Sym}, we can rewrite \eqref{E:I2} in terms of a symmetric square $L$-value.
\begin{corollary}\label{C:Sym}
We have
\begin{equation*}
I(-32)+2I(4)=I(-32)+I(16)=54L(\Sym^2 E_{1/3},2).
\end{equation*}
\end{corollary} 

\section{Final remarks} \label{S:conclude}
In the final section, we give some possible directions for future research. First, due to the fact that the integral $I(t)$ arises from a computation in physics, it would be interesting to understand a physical interpretation of the identities \eqref{E:I2} and \eqref{E:I3}. On the other hand, one might be tempted to find other identities of similar type for $I(t)$. When  $t =64,8,-8,$ the $K3$ surface $X_t$ is singular and the order of complex multiplication corresponding to $t$ has discriminant $D=-15,-8,-24,$ respectively. In theory, for each $t$ given above, one can evaluate $I(t)$ at the corresponding CM point $\tau$ using \eqref{E:Imod} and identify the term $\varpi_2(-1/6\tau)$ with a critical $L$-value of some weight $3$ cusp form using \cite[Thm. 5]{RWZ}. However, following the argument in the proofs of our main theorems, we also have in the expression of $I(t)$ a combination of special values of the function $F_3$, whose explicit evaluations are not known in general. For instance, if $\tau=(-3+\sqrt{-15})/6$, then we have $t(\tau)=64$. Note that the integral $I(t)$ may not converge when $t>16$. If we consider the expression \eqref{E:Imod} as an analytic continuation of $I(t)$, then we find that 
\[\Re(I(64))=\frac{L(f,2)}{8\pi^2}\left(\frac{43\sqrt{15}\pi^3}{15}-45\zeta(3)-45\left(3F_3\left(\frac{\sqrt{-15}}{3}\right)-F_3(\sqrt{-15})\right)\right).\] 

One might also consider finding a direct relationship between the integral $I(t)$ and a critical value of the symmetric square $L$-function of $E_{r(t)}$, as suggested by Theorem~\ref{T:Sym} and Corollary~\ref{C:Sym}. Unfortunately, for all but finitely many values of $t$, $r(t)$ is irrational and the theory of symmetric square $L$-functions of elliptic curves defined over number fields has not yet been completely established. (For instance, the elliptic curve $E_{r(1)}$ is defined over $\mathbb{Q}(\sqrt{5})$.) Goncharov briefly explained the connection between Feynman integrals and special values of symmetric power $L$-functions from a polylogarithmic point of view in \cite[\S 3.5.7]{Gon}.

We close this article by pointing out some parallel results on another object, the Mahler measure of the polynomial $P_t$, which is defined analogously to the integral $I(t)$. Indeed, we first discovered the identities \eqref{E:I2} and \eqref{E:I3} from our numerical computation based on some known results on Mahler measures. Recall that, for a nonzero Laurent polynomial $P\in\mathbb{C}[x_1^{\pm 1},\ldots,x_n^{\pm 1}]$, the (logarithmic) Mahler measure of $P$ is 
\[m(P)=\int_0^1\cdots \int_0^1 \log |P(e^{2\pi i \theta_1},\ldots,e^{2\pi i \theta_n})|\,d\theta_1\cdots d\theta_n.\]
It has been proved that the Mahler measures of certain three-variable polynomials are expressible in terms of non-critical $L$-values. For example, Bertin \cite{Ber06a,Ber06b} proved that, for some integral values of $t$ which are corresponding to singular $K3$ surfaces, the Mahler measures of 
\begin{align*}
Q_t:&=x_1+x_1^{-1}+x_2+x_2^{-1}+x_3+x_3^{-1}-t,\\
R_t:&=(1+x_1+x_2+x_3)(x_1x_2+x_2x_3+x_1x_3+x_1x_2x_3)-tx_1x_2x_3,
\end{align*}
are rational linear combinations of $\sqrt{N}L(h,3)/\pi^3$ and $\sqrt{D}L(\chi,2)/\pi$ for some newform $h\in \mathrm{S}_3(\Gamma_1(N))$ and some odd Dirichlet character $\chi$ of conductor $D$. By a substitution $(x_1,x_2,x_3)\rightarrow (x_1/x_3,x_3/x_2,x_2)$ in $m(R_t)$ and \cite[Lem. 7]{Smy}, we have that $m(P_t)=m(R_t),$ which we shall denote by $m(t)$. As a consequence, Bertin's formulas for $m(R_t)$ given in \cite[Thm. 1]{Ber06a} and \cite[\S 4.3]{Ber06b} become equivalent to 
\begin{align*}
m(16)=4m(4)&=\frac{48\sqrt{3}}{\pi^3}L(g,3),\\
m(4)&=2(m(-32)-2m(-2)),
\end{align*}
which are clearly reminiscent of the identities in Theorem~\ref{T:main2}. The author also showed further in \cite[Cor. 3.2]{Sam} that 
\begin{align*}
m(-32)&=\frac{48\sqrt{3}}{\pi^3}L(g,3)+\frac{4}{\pi}L(\chi_{-4},2),\\
m(-2)&=\frac{21\sqrt{3}}{\pi^3}L(g,3)+\frac{2}{\pi}L(\chi_{-4},2).
\end{align*}
These formulas lead us to believe that each integral $I(t)$ in Theorem~\ref{T:main2} can also be written in terms of a special $L$-value and some `meaningful' quantity. Given these examples and \eqref{E:Bro}, one should expect that $m(1)$ is also expressible in terms of $L(f,3).$ In contrast, by a result of Bertin \cite{Ber12}, we have
\[m(1)=\frac{6\sqrt{3}}{5\pi}L(\chi_{-3},2).\]
A non-critical $L$-value of $f$, on the other hand, conjecturally appears in the Mahler measure of a linear polynomial of four variables. 
\begin{conjecture}[F. Rodriguez Villegas \cite{BLVD}]
We have
\[m(1+x_1+x_2+x_3+x_4)\stackrel{?}=6\left(\frac{\sqrt{-15}}{2\pi i}\right)^5L(f,4).\]
\end{conjecture}
\noindent The (desingularized) variety $1+x_1+x_2+x_3+x_4=0$ is not a surface, but the complete intersection of this hypersurface and $1+x_1^{-1}+x_2^{-1}+x_3^{-1}+x_4^{-1}=0$ is compactified to a singular $K3$ surface, whose $L$-function is $L(f,s)$ \cite{PTV}. In fact, by a transformation given in \cite[\S 7.1.1]{BKV}, it can be seen that this $K3$ surface is isomorphic to $X_1$. 

Let us elucidate a link between $I(t)$ and $m(t)$ here. Recall from \cite[Thm. 2.2.1]{BKV} that $I(t)$ satisfies the nonhomogeneous differential equation $\mathcal{L}_t^3 I(t) =-24,$ while the derivative $m'(t)$ is a solution of the associated homogeneous equation. Hence, parametrizing $t$ by \eqref{E:t} we can write $m'(t)$ in terms of $\varpi_1(\tau).$ We also have from the proofs of \cite[Thm. 2.3.2]{BKV} and \cite[Thm. 1.1]{Ber06b} that
\begin{align*}
I(t(\tau))&= \varpi_1(\tau)\left(40\pi^2 \tau^2+\frac{1}{2}\int_1^q\left(\log \frac{\hat{q}}{q}\right)^2\sigma(\hat{q})d\log\hat{q}\right),\\
m(t(\tau)) &= \frac{1}{24}\int_1^q \sigma(\hat{q})d\log \hat{q},
\end{align*}
where $\sigma(q)=\frac{1}{5}(-E_4(\tau)+16E_4(2\tau)+9E_4(3\tau)-144E_4(6\tau)).$ 

The connection between Feynman integrals and Mahler measures seems even more obscure in the two-variable cases. For instance, consider the so-called sunset integral
\[J(t)=\int_{x_1,x_2\ge 0} \frac{1}{(1+x_1+x_2)(1+x_1^{-1}+x_2^{-1})-t}\,\frac{d x_1}{x_1}\frac{d x_2}{x_2},\]
studied in \cite{BV}. Note that when $t\ne 0$ the zero locus of $S_t:=(1+x_1+x_2)(1+x_1^{-1}+x_2^{-1})-t$ is the curve $E_{1/t}$ introduced in Section~\ref{S:SI}, but, to our knowledge, no $J(t)$ is known to be related to an $L$-value of $E_{1/t}.$ (This may be regarded as a part of Broadhurst's remark in \cite{Bro14}: ``The absence of weight $2$ examples is remarkable: does quantum field theory avoid Birch and Swinnerton-Dyer?''.) On the other hand, Rogers \cite{Rog} verified numerically that for many integral values of $t$, $m(S_t)$ is a rational multiple of $L(E_{1/t},2)/\pi^2.$ For example, he found that 
\begin{equation*}
m(S_8)=6m(S_2)=\frac{3}{5}m(S_{-7})=\frac{21}{\pi^2}L(E,2),
\end{equation*}
where $E$ is an elliptic curve of conductor $14$, which have been successfully proved by Mellit \cite{Mellit}. Despite the absence of weight $2$ evidence for $J(t)$, we have found from our numerical computation that \[J(8)\stackrel{?}=2J(2).\] (Curiously, the integral $J(-7)$ does not appear to be a rational multiple of the other two.) In summary, it would be interesting to gain deeper understanding about how Feynman integrals and Mahler measures are related.\\

\noindent\textbf{Acknowledgements}\\
The author would like to express his gratitude to Scott Ahlgren for many helpful discussions and insightful suggestions on an earlier version of this manuscript. The author would also like to thank Matt Kerr for answering his question about proof details in \cite{BKV} and Armin Straub for bringing relevant results on short uniform random walks \cite{BSWZ} to his attention. Last but not least, the author is grateful to the referees for directing him to Mellit's result \cite{Mellit} and other valuable suggestions which help to improve the exposition of this paper.

\bibliographystyle{amsplain}

\end{document}